\documentclass{amsproc}%
\usepackage{amsfonts}
\usepackage{amsmath}
\usepackage{amssymb}
\usepackage{graphicx}%
\providecommand{\U}[1]{\protect \rule{.1in}{.1in}}
\theoremstyle{plain}

\newtheorem{lemma}{Lemma}

\newtheorem{theorem}{Theorem}
\numberwithin{equation}{section}
\begin{document}
\title[Tracially Nuclear C*-algebras]{A Characterization of Tracially Nuclear C*-algebras}
\author{Don Hadwin}
\address{Mathematics Department, University of New Hampshire}
\email{don@unh.edu}
\author{Weihua Li}
\address{Science and Mathematics Department, Columbia College Chicago}
\email{wli@colum.edu}
\author{Wenjing Liu}
\address{Mathematics Department, University of New Hampshire}
\email{wenjingtwins87@gmail.com}
\author{Junhao Shen}
\address{Mathematics Department, University of New Hampshire}
\email{Junhao.Shen@unh.edu}
\thanks{Supported by a Collaboration Grant from the Simons Foundation}
\thanks{Supported by a Faculty Development Grant from Columbia College Chicago }
\thanks{Supported by a Dissertation Year Fellowship from the University of New Hampshire}
\dedicatory{Dedicated to the memory Uffe Haagerup.}
\begin{abstract}
We give two characterizations of tracially nuclear C*-algebras. The first is
that the finite summand of the second dual is hyperfinite. The second is in
terms of a variant of the weak* uniqueness property. The necessary condition
holds for all tracially nuclear C*-algebras. When the algebra is separable, we
prove the sufficiency.

\end{abstract}
\maketitle

\section{Introduction}

In \cite{HL} a unital C*-algebra $\mathcal{A}$ was defined to be
\emph{tracially nuclear} if, for every tracial state $\tau$ on $\mathcal{A}$,
if $\pi_{\tau}$ is the GNS representation for $\tau$, then $\pi_{\tau}\left(
\mathcal{A}\right)  ^{\prime \prime}$ is hyperfinite. Tracially nuclear
algebras also played a role in the theory of tracially stable C*-algebras
\cite{HS}. In this paper we give two new characterizations of tracially
nuclear C*-algebras, the first (Theorem \ref{2dual}) in terms of the second
dual of the algebra, and the second (Theorem \ref{half2}) in terms of weak*
approximate equivalence of representations into finite von Neumann algebras.
In one direction, we show (Theorem \ref{half1}) that if $\mathcal{A}$ is any
tracially nuclear C*-algebra, and $\mathcal{M}$ is any finite von Neumann
algebra, then the rank condition in \cite{DH} on two representations $\pi
,\rho:\mathcal{A}\rightarrow \mathcal{M}$ implies a strong version of weak*
approximate equivalence of $\pi$ and $\rho$. When $\mathcal{A}$ is separable
we prove the converse (Theorem \ref{half2}). Thus the second characterization
is an analogue of the characterization if nuclearity given in \cite{PWN}.

When $\mathcal{A}$ is separable, we only need to check $\pi_{\tau}\left(
\mathcal{A}\right)  ^{\prime \prime}$ is hyperfinite when $\tau$ is an
\emph{infinite-dimensional} \emph{factor state}, i.e., $\pi_{\tau}\left(
\mathcal{A}\right)  ^{\prime \prime}$ is a $II_{1}$ factor von Neumann algebra.

\begin{lemma}
\label{factor}Suppose $\mathcal{A}$ is a separable unital C*-algebra. Then
$\mathcal{A}$ is tracially nuclear if and only if, for every
infinite-dimensional factor tracial state $\tau$ on $\mathcal{A}$, $\pi_{\tau
}\left(  \mathcal{A}\right)  ^{\prime \prime}$ is hyperfinite.
\end{lemma}

\begin{proof}
We let $N=\pi_{\tau}\left(  \mathcal{A}\right)  ^{\prime \prime}$. Since
$\mathcal{A}$ is separable, $N$ acts on a separable Hilbert space. Using the
central decomposition we can write $N=\int_{\Omega}^{\oplus}N_{\omega}%
d\mu \left(  \omega \right)  $ where each $N_{\omega}$ is a factor von Neumann
algebra, and we can write $\pi_{\tau}=\int_{\Omega}^{\oplus}\pi_{\omega}%
d\mu \left(  \omega \right)  $ and $\tau=\int_{\Omega}^{\oplus}\tau_{\omega}%
d\mu \left(  \omega \right)  $ with each $\tau_{\omega}$ a factor state, each
$\pi_{\omega}=\pi_{\tau_{\omega}},$ and each $\pi_{\tau_{\omega}}\left(
\mathcal{A}\right)  ^{\prime \prime}=N_{\omega}$. Since $N$ is hyperfinite if
and only if almost every $N_{\omega}$ is hyperfinite, and since every
finite-dimensional factor is hyperfinite, the lemma is proved.
\end{proof}

\bigskip

\section{The second dual $\mathcal{A}^{\# \#}$}

If $\mathcal{R}\subset B\left(  H\right)  $ is a finite von Neumann algebra,
then we can write $H=\sum_{\gamma \in \Gamma}^{\oplus}H_{\gamma}$ and
$\mathcal{R}=\sum_{\gamma \in \Gamma}^{\oplus}\mathcal{R}_{\gamma}$, where each
$\mathcal{R}_{\gamma}\subset B\left(  H_{\gamma}\right)  $ has a faithful
normal tracial state $\tau_{\gamma}$. We can extend each $\tau_{\gamma}$ to a
tracial state on $\mathcal{R}$ by saying if $T=\sum_{\lambda \in \Gamma}%
^{\oplus}T_{\lambda},$ then $\tau_{\gamma}\left(  T\right)  =\tau_{\gamma
}\left(  T_{\gamma}\right)  $. Each $\tau_{\gamma}$ gives a seminorm
$\left \Vert T\right \Vert _{2,\gamma}=\tau_{\gamma}\left(  T^{\ast}T\right)
^{1/2}$. It is a simple fact that on bounded subsets of $\mathcal{R}$, the
strong (SOT) and $\ast$-strong ($\ast$-SOT) operator topologies coincide and
are generated by the family $\left \{  \left \Vert \cdot \right \Vert _{2,\gamma
}:\gamma \in \Gamma \right \}  $. Thus a bounded net $\left \{  T_{n}\right \}  $ in
$\mathcal{R}$ converges in SOT or $\ast$-SOT to $T\in \mathcal{R}$ if and only
if, for every $\gamma \in \Gamma$,%
\[
\left \Vert T_{n}-T\right \Vert _{2,\gamma}\rightarrow0.
\]
Also every von Neumann algebra $\mathcal{R}$ can uniquely be decomposed into a
direct sum $\mathcal{R}=\mathcal{R}_{f}\oplus \mathcal{R}_{i}$, where
$\mathcal{R}_{f}$ is a finite von Neumann algebra and $\mathcal{R}_{i}$ has no
finite direct summands. Equivalently, $\mathcal{R}_{i}$ has no normal tracial
states. Relative to this decomposition, we write $Q_{f,\mathcal{R}}=1\oplus0$.

If $\mathcal{A}$ is a unital C*-algebra, then $\mathcal{A}^{\# \#}$ is a von
Neumann algebra, and, using the universal representation, we can assume
$\mathcal{A}\subset \mathcal{A}^{\# \#}\subset B\left(  \mathcal{H}\right)  $
where the weak* topology on $\mathcal{A}^{\# \#}$ coincides with the weak
operator topology, so that $\mathcal{A}^{\prime \prime}=\mathcal{A}^{\# \#}$.
Moreover, for every von Neumann algebra $\mathcal{R}$ and every unital $\ast
$-homomorphism $\pi:\mathcal{A}\rightarrow \mathcal{R}$, there is a weak*-weak*
continuous unital $\ast$-homomorphism $\hat{\pi}:\mathcal{A}^{\#
\#}\rightarrow \mathcal{R}$ such that $\hat{\pi}|_{\mathcal{A}}=\pi$. Moreover,
ker$\hat{\pi}$ being a weak* closed two-sided ideal in $\mathcal{A}^{\# \#}$
has the form%
\[
\ker \hat{\pi}=\left(  1-P_{\pi}\right)  \mathcal{A}^{\# \#},\text{ with
}P_{\pi}=P_{\pi}^{2}=P_{\pi}^{\ast}\in \mathcal{Z}\left(  \mathcal{A}^{\#
\#}\right)  ,
\]
where $\mathcal{Z}\left(  \mathcal{M}\right)  $ denotes the center of a von
Neumann algebra $\mathcal{M}$. Thus%
\[
\mathcal{A}^{\# \#}=P_{\pi}\mathcal{A}^{\# \#}\oplus \ker \hat{\pi}\text{ .}%
\]
The following theorem contains our first characterization of tracially nuclear C*-algebras.

\begin{theorem}
\label{2dual}If $\mathcal{A}$ is a unital C*-algebra, then

\begin{enumerate}
\item For every unital $\ast$-homomorphism $\pi:\mathcal{A}\rightarrow
\mathcal{M}$ with $\mathcal{M}$ a finite von Neumann algebra, $P_{\pi}\leq
Q_{f,\mathcal{A}^{\# \#}}$.

\item $\mathcal{A}$ is tracially nuclear if and only if $\  \left(
\mathcal{A}^{\# \#}\right)  _{f}$ is a hyperfinite von Neumann algebra.
\end{enumerate}
\end{theorem}

\begin{proof}
$\left(  1\right)  $. Assume, via contradiction, $\hat{\pi}\left(
1-Q_{f,\mathcal{A}^{\# \#}}\right)  \neq0$. Since $\mathcal{M}$ is finite,
there is a normal tracial state $\tau$ on $\mathcal{M}$ such that
\[
s=\tau \left(  \hat{\pi}\left(  1-Q_{f,\mathcal{A}^{\# \#}}\right)  \right)
\neq0.
\]
Hence the map $\gamma:\left(  \mathcal{A}^{\# \#}\right)  _{i}\rightarrow
\mathbb{C}$ defined by%
\[
\gamma \left(  T\right)  =\frac{1}{s}\hat{\pi}\left(  0\oplus T\right)
\]
is a faithful normal tracial state on $\left(  \mathcal{A}^{\# \#}\right)
_{i}$, which is a contradiction. Thus
\[
\hat{\pi}\left(  1-Q_{f,\mathcal{A}^{\# \#}}\right)  =0,
\]
which means that $P_{\pi}\leq Q_{f,\mathcal{A}^{\# \#}}$.

$\left(  2\right)  $. Suppose $\mathcal{A}$ is tracially nuclear.
$(\mathcal{A}^{\# \#})_{f}=\Sigma_{\lambda \in \Lambda}^{\oplus}({\mathcal{R}%
}_{\lambda},\tau_{\lambda})$, where $\tau_{\lambda}$ is a faithful normal
tracial state on ${\mathcal{R}}_{\lambda}.$ Then $\mathcal{A}^{\#
\#}=(\mathcal{A}^{\# \#})_{f}\oplus \left(  \mathcal{A}^{\# \#}\right)
_{i}=\Sigma_{\lambda \in \Lambda}^{\oplus}{\mathcal{R}}_{\lambda}\oplus \left(
\mathcal{A}^{\# \#}\right)  _{i}$relative to $\mathcal{H}=\sum_{\lambda
\in \Lambda}^{\oplus}\mathcal{H}_{\lambda}\oplus \mathcal{H}_{i}$. Viewing
$\mathcal{A}\subset \mathcal{A}^{\# \#}$, we let $\pi_{\lambda}:\mathcal{A}%
\rightarrow \mathcal{R}_{\lambda}$ be defined by $\pi_{\lambda}\left(
A\right)  =A|_{\mathcal{H}_{\lambda}}$. Then $\psi_{\lambda}=\tau_{\lambda
}\circ \pi_{\lambda}$ is a tracial state on $\mathcal{A}$ and, since
$\mathcal{A}$ is weak*-dense in $\mathcal{A}^{\# \#}$, $\pi_{\psi_{\lambda}%
}\left(  \mathcal{A}\right)  ^{-\text{\textrm{weak*}}}=\mathcal{R}_{\lambda}$.
Since $\mathcal{A}$ is tracially nuclear, $\mathcal{R}_{\lambda}$ must be
hyperfinite. Hence, $({\mathcal{A}}^{\sharp \sharp})_{f}=\Sigma_{\lambda
\in \Lambda}^{\oplus}{\mathcal{R}}_{\lambda}$ is hyperfinite.

Conversely, suppose $\left(  \mathcal{A}^{\# \#}\right)  _{f}$ is hyperfinite,
and suppose $\tau$ is a tracial state on $\mathcal{A}$. Since $\pi_{\tau
}\left(  \mathcal{A}\right)  ^{\prime \prime}$ has a faithful normal tracial
state, it must be finite. Thus $P_{\pi_{\tau}}\leq$ $Q_{f,\mathcal{A}^{\# \#}%
}$. This means that $P_{\pi_{\tau}}\mathcal{A}^{\# \#}$ is a direct summand of
$\left(  \mathcal{A}^{\# \#}\right)  _{f}$, and is therefore hyperfinite. But
this summand is isomorphic to $\pi_{\tau}\left(  \mathcal{A}\right)
^{\prime \prime}.$ Thus $\mathcal{A}$ is tracially nuclear.
\end{proof}

\section{Weak* approximate equivalence in finite von Neumann algebras}

Suppose $\mathcal{A}$ is a unital C*-algebra, $\mathcal{R}$ is a von Neumann
algebra and $\pi,\rho:\mathcal{A}\rightarrow \mathcal{R}$ are unital $\ast
$-homomorphisms. Following \cite{PWN}, $\pi$ and $\rho$ are \emph{weak*
approximately equivalent} if there are nets $\left \{  U_{\lambda}\right \}  $
and $\left \{  V_{\lambda}\right \}  $ of unitary operators in $\mathcal{R}$
such that, for every $A\in \mathcal{A}$,
\[
U_{\lambda}^{\ast}\pi \left(  A\right)  U_{\lambda}\overset
{\text{\textrm{weak*}}}{\rightarrow}\rho \left(  A\right)  \text{ and
}V_{\lambda}^{\ast}\rho \left(  A\right)  V_{\lambda}\overset
{\text{\textrm{weak*}}}{\rightarrow}\pi \left(  A\right)  \text{.}%
\]
It was observed in \cite{PWN} that it follows that the convergence above
actually occurs in the $\ast$-strong operator topology ($\ast$-SOT).

Suppose $\mathcal{M}$ is a von Neumann algebra and $T\in \mathcal{M}$.
Following \cite{DH}, $\mathcal{M}$-rank$\left(  T\right)  $ is defined to be
the Murray von Neumann equivalence class in $\mathcal{M}$ of the projection
onto the closure of the range of $T$. In \cite{PWN} it was shown that if
$\mathcal{A}$ is a separable nuclear C*-algebra and $\mathcal{M}$ is a von
Neumann algebra acting on a separable Hilbert space, then two unital $\ast
$-homomorphisms $\pi,\rho:\mathcal{A}\rightarrow \mathcal{M}$ are weak*
approximately equivalent if and only if, $\left(  \mathcal{M}\text{-rank}%
\right)  \circ \pi=\left(  \mathcal{M}\text{-rank}\right)  \circ \rho$. They
also proved that this property for $\mathcal{A}$ is equivalent to nuclearity.

The following result is from \cite{DW}.For completeness we include a short proof.

\begin{lemma}
\label{H-D} \cite{DW}Suppose $a=a^{\ast}$ in $\mathcal{B}(H)$, $0\leq a\leq1$
and ${\mathcal{C}}_{0}^{\ast}(a)$ is the norm-closure of $\{p(a),p\in
\mathbb{C}\left[  z\right]  ,$ $p(0)=0\}$. Suppose $\mathcal{M}$ is a finite
von Neumann algebra with a center-valued trace $\Phi:{\mathcal{M}}%
\rightarrow{\mathcal{Z}(M)}$, and $\pi,\rho:{\mathcal{C}}_{0}^{\ast
}(a)\rightarrow{\mathcal{M}}$ are *-homomorphisms. Then the following are equivalent:

(1). $\forall x\in{\mathcal{C}}_{0}^{\ast}(a)$, ${\mathcal{M}}$-$rank\  \pi
(x)={\mathcal{M}}$-$rank\  \rho(x),$

(2). $\Phi \circ \pi=\Phi \circ \rho$.
\end{lemma}

\begin{proof}
$\left(  1\right)  \Rightarrow \left(  2\right)  $. We can extend $\pi$ and
$\rho$ to weak*-weak* continuous *-homomorphisms $\hat{\pi},\hat{\rho
}:{\mathcal{C}}_{0}^{\ast}(a)^{\# \#}\rightarrow \mathcal{M}$. Suppose
$x\in{\mathcal{C}}_{0}^{\ast}(a)$ and $0\leq x\leq1$. Suppose $0<\alpha<1$ and
define $f_{\alpha}:\left[  0,1\right]  \rightarrow \left[  0,1\right]  $ by
\[
f\left(  t\right)  =dist\left(  t,\left[  0,\alpha \right]  \right)  .
\]
Since $f\left(  0\right)  =0,$ we see that $f\left(  x\right)  \in \mathcal{A}%
$, and $\chi_{(\alpha,1]}\left(  x\right)  =$ weak*-$\lim_{n\rightarrow \infty
}f\left(  x\right)  ^{1/n}\in \mathcal{A}^{\# \#}$, so
\[
\hat{\pi}\left(  \chi_{(\alpha,1]}\left(  x\right)  \right)  \text{, and }%
\hat{\rho}\left(  \chi_{(\alpha,1]}\left(  x\right)  \right)  \text{ }%
\]
are the range projections for $\pi \left(  f\left(  x\right)  \right)  $ and
$\rho \left(  f\left(  x\right)  \right)  $, respectively. Since
\[
{\mathcal{M}}-rank\  \pi(f\left(  x\right)  )={\mathcal{M}}-rank\  \rho(f\left(
x\right)  ),
\]
we see that $\hat{\rho}\left(  \chi_{(\alpha,1]}\left(  x\right)  \right)  $
and $\hat{\pi}\left(  \chi_{(\alpha,1]}\left(  x\right)  \right)  $ are Murray
von Neumann equivalent.Hence
\[
\Phi \left(  \hat{\pi}\left(  \chi_{(\alpha,1]}\left(  x\right)  \right)
\right)  =\Phi \left(  \hat{\rho}\left(  \chi_{(\alpha,1]}\left(  x\right)
\right)  \right)  .
\]
Thus, Suppose $0<\alpha<\beta<1$. Since $\chi_{(\alpha,\beta]}=\chi
_{(\alpha,1]}-\chi_{(\beta,1]},$ we see that%
\[
\Phi \left(  \hat{\pi}\left(  \chi_{(\alpha,\beta]}\left(  x\right)  \right)
\right)  =\Phi \left(  \hat{\rho}\left(  \chi_{(\alpha,\beta]}\left(  x\right)
\right)  \right)  .
\]
Thus, for all $n\in \mathbb{N}$,%
\[
\Phi \left(  \hat{\pi}\left(  \sum_{k-1}^{n-1}\frac{k}{n}\chi_{(\frac{k}%
{n},\frac{k+1}{n}]}\left(  x\right)  \right)  \right)  =\Phi \left(  \hat{\rho
}\left(  \sum_{k-1}^{n-1}\frac{k}{n}\chi_{(\frac{k}{n},\frac{k+1}{n}]}\left(
x\right)  \right)  \right)  .
\]
Since, for every $n\in \mathbb{N}$,
\[
\left \Vert x-\sum_{k-1}^{n-1}\frac{k}{n}\chi_{(\frac{k}{n},\frac{k+1}{n}%
]}\left(  x\right)  \right \Vert \leq1/n\text{,}%
\]
it follows that
\[
\Phi \left(  \pi \left(  x\right)  \right)  =\Phi \left(  \hat{\pi}\left(
x\right)  \right)  =\Phi \left(  \hat{\rho}\left(  x\right)  \right)
=\Phi \left(  \rho \left(  x\right)  \right)  .
\]
Since $\mathcal{A}$ is the linear span of its positive contractions,
$\Phi \circ \pi=\Phi \circ \rho$.

$\left(  2\right)  \Rightarrow \left(  1\right)  $. Since $\Phi,$ $\hat{\pi}$
and $\hat{\rho}$ are weak*-weak* continuous, it follows that $\Phi \circ
\hat{\pi}=\Phi \circ \hat{\rho}$, so we see, for any $x\in{\mathcal{C}}%
_{0}^{\ast}(a)$ that%
\[
\Phi \left(  \hat{\pi}\left(  \chi_{(0,\infty)}\left(  \left \vert x\right \vert
\right)  \right)  \right)  =\Phi \left(  \hat{\rho}\left(  \chi_{(0,\infty
)}\left(  \left \vert x\right \vert \right)  \right)  \right)  ,
\]
which implies that $\chi_{(0,\infty)}\left(  \left \vert \pi \left(  x\right)
\right \vert \right)  $ and $\chi_{(0,\infty)}\left(  \left \vert \rho \left(
x\right)  \right \vert \right)  .$ Thus ${\mathcal{M}}$-$rank\  \pi
(x)={\mathcal{M}}$-$rank\  \rho(x)$.
\end{proof}

The following lemma is from \cite{DH}.

\begin{lemma}
\label{H-D1} \cite{DH}Suppose ${\mathcal{B}}=\Sigma_{m=1}^{t}{\mathcal{M}%
}_{k_{m}}(\mathbb{C})$ with matrix units $e_{i,j,m}$, $\mathcal{D}$ is a
unital C*-algebra, and $\pi,\rho:{\mathcal{B}}\rightarrow \mathcal{D}$ are
unital *-homomorphisms such that $\pi(e_{i,i,m})\sim \rho(e_{i,i,m})$. Then,
there exists a unitary $w\in \mathcal{D}$ such that $\pi \left(  \cdot \right)
=w^{\ast}\rho \left(  \cdot \right)  w.$
\end{lemma}

\begin{theorem}
\label{half1}Suppose $\mathcal{A}$ is a unital tracially nuclear C$^{\text{*}%
}$-algebra, $\mathcal{M}$ is a finite von Neumann algebra with center-valued
trace $\Phi$, and $\pi,\rho:{\mathcal{A}}\rightarrow{\mathcal{M}}$ are unital
*-homomorphisms. The following are equivalent:

\begin{enumerate}
\item For every $a\in \mathcal{A}$, ${\mathcal{M}}$-rank $\pi(a)={\mathcal{M}}%
$-rank $\rho(a)$.

\item $\Phi \circ \pi=\Phi \circ \rho.$

\item The representations $\pi$ and $\rho$ are weak* approximately equivalent.

\item There is a net $\{U_{n}\}$ of unitary operators in $\mathcal{M}$ such
that, for every $a\in \mathcal{A}^{\# \#}$,

\begin{enumerate}
\item $U_{n}\pi(a)U_{n}^{\ast}\rightarrow \rho(a)$ in the $\ast$strong operator
topology, and

\item $U_{n}^{\ast}\rho(a)U_{n}\rightarrow \pi(a)$ in the $\ast$strong operator topology.
\end{enumerate}
\end{enumerate}
\end{theorem}

\begin{proof}
Clearly, $\left(  4\right)  \Rightarrow \left(  3\right)  \Rightarrow \left(
2\right)  $.

$\left(  1\right)  \Leftrightarrow \left(  2\right)  $. This is proved in
\ref{H-D}.

$\left(  2\right)  \Rightarrow \left(  4\right)  $. Let $\hat{\pi},\hat{\rho
}:{\mathcal{A}}^{\sharp \sharp}\rightarrow{\mathcal{M}}$ be the weak*-weak*
continuous extensions of $\pi$ and ${\rho}$, respectively. Since $\Phi$ is
weak*-weak* continuous, we see that $\Phi \circ \hat{\pi}=\Phi \circ \hat{\rho}$.
Since $\mathcal{M}$ is finite, $\mathcal{M}$ can be written as ${\mathcal{M}%
}=\Sigma_{\gamma \in \Gamma}^{\oplus}({\mathcal{M}}_{\gamma},\beta_{\gamma})$,
where $\beta_{\gamma}$ is a faithful normal tracial state of ${\mathcal{M}%
}_{\gamma}$. Similarly, we can write $({\mathcal{A}}^{\sharp \sharp}%
)_{f}=\Sigma_{\lambda \in \Lambda}^{\oplus}({\mathcal{R}}_{\lambda}%
,\tau_{\lambda})$ where $\tau_{\lambda}$ is a faithful normal tracial state on
$\mathcal{R}_{\lambda}$ for each $\lambda \in \Lambda$. Thus $\mathcal{A}^{\#
\#}=\Sigma_{\lambda \in \Lambda}^{\oplus}({\mathcal{R}}_{\lambda},\tau_{\lambda
})\oplus \left(  A^{\# \#}\right)  _{i}$. If $S\in \mathcal{M}$ and
$T\in \mathcal{A}^{\# \#}$, we write%
\[
S=\sum_{\gamma \in \Gamma}S\left(  \gamma \right)  \text{ and }T=\sum_{\lambda
\in \Lambda}T\left(  \lambda \right)  \oplus T\left(  i\right)  \text{.}%
\]
Since $\mathcal{M}$ is finite, we know from \ref{2dual} that $\hat{\pi}\left(
Q_{f,\mathcal{A}^{\# \#}}\right)  =\hat{\rho}\left(  Q_{f,\mathcal{A}^{\# \#}%
}\right)  =1$. We also know that $\hat{\pi}$ and $\hat{\rho}$ are continuous
in the strong operator topology. Thus if $\left \{  T_{j}\right \}  $ is a
norm-bounded net in $\mathcal{A}^{\# \#}$ and $T\in \mathcal{A}^{\# \#},$ and
$T_{j}Q_{f,\mathcal{A}^{\# \#}}\rightarrow TQ_{f,\mathcal{A}^{\# \#}}$ in the
strong operator topology, then $\hat{\pi}\left(  T_{j}\right)  =\hat{\pi
}\left(  T_{j}Q_{f,\mathcal{A}^{\# \#}}\right)  \rightarrow \hat{\pi}\left(
T\right)  $ and $\hat{\rho}\left(  T_{j}\right)  \rightarrow \hat{\rho}\left(
T\right)  $ in the strong operator topology. This means that, if, for every
$\lambda \in \Lambda$, we have $\left \Vert T_{j}\left(  \lambda \right)
-T\left(  \lambda \right)  \right \Vert _{2,\tau_{\lambda}}\rightarrow0$, then,
for every $\gamma \in \Gamma$, we have
\[
\left \Vert \hat{\pi}\left(  T_{j}\right)  \left(  \gamma \right)  -\hat{\pi
}\left(  T\right)  \left(  \gamma \right)  \right \Vert _{2,\beta_{\gamma}%
}\rightarrow0\text{ and }\left \Vert \hat{\rho}\left(  T_{j}\right)  \left(
\gamma \right)  -\hat{\rho}\left(  T\right)  \left(  \gamma \right)  \right \Vert
_{2,\beta_{\gamma}}\rightarrow0\text{ .}%
\]

Suppose $A\subset$ \textrm{ball}$\left(  \mathcal{A}^{\# \#}\right)  $ is
finite, $L\subset \Lambda$ is finite and $\varepsilon>0.$ Then there is a
$\delta>0$ and a finite subset $G\subset \Gamma$ such that, if $T\in A,$
$S\in2$\textrm{ball}$\left(  \mathcal{A}^{\# \#}\right)  $ and, for every
$\lambda \in L$, we have $\left \Vert T\left(  \lambda \right)  -S\left(
\lambda \right)  \right \Vert _{2,\tau_{\lambda}}<\delta$, then
\[
\sum_{\gamma \in G}\left[  \left \Vert \hat{\pi}\left(  S\right)  \left(
\gamma \right)  -\hat{\pi}\left(  T\right)  \left(  \gamma \right)  \right \Vert
_{2,\beta_{\gamma}}+\left \Vert \hat{\rho}\left(  S\right)  \left(
\gamma \right)  -\hat{\rho}\left(  T\right)  \left(  \gamma \right)  \right \Vert
_{2,\beta_{\gamma}}\right]  <\varepsilon/37\text{ .}%
\]
Since $\mathcal{A}$ is tracially nuclear, we know that, for every $\lambda
\in \Lambda$, $\mathcal{R}_{\lambda}$ is hyperfinite. Thus, for each
$\lambda \in L$, there is a finite-dimensional unital C*-subalgebra
$\mathcal{B}_{\lambda}\subset \mathcal{R}_{\lambda}$ such that, for each $S\in
A$, there is a $B_{\lambda,S}\in \mathcal{B}_{\lambda}$ such that $\left \Vert
B_{\lambda,S}\right \Vert \leq \left \Vert S\left(  \lambda \right)  \right \Vert $
and $\left \Vert S\left(  \lambda \right)  -B_{\lambda,S}\right \Vert
_{2,\tau_{\lambda}}<\delta$. Then $\mathcal{B}=\sum_{\lambda \in L}^{\oplus
}\mathcal{B}_{\lambda}$ is a finite-dimensional C*-subalgebra of
$\mathcal{A}^{\# \#}$, and, for each $S\in A$, we define $B_{s}=\sum
_{\lambda \in L}^{\oplus}B_{\lambda,S}\in \mathcal{B}$. It follows that
\[
\sum_{S\in A}\sum_{\gamma \in G}\left[  \left \Vert \hat{\pi}\left(  S\right)
\left(  \gamma \right)  -\hat{\pi}\left(  B_{S}\right)  \left(  \gamma \right)
\right \Vert _{2,\beta_{\gamma}}+\left \Vert \hat{\rho}\left(  S\right)  \left(
\gamma \right)  -\hat{\rho}\left(  B_{S}\right)  \left(  \gamma \right)
\right \Vert _{2,\beta_{\gamma}}\right]  <\varepsilon/37\text{ .}%
\]
We know from $\Phi \circ \hat{\pi}=\Phi \circ \hat{\rho}$ and Lemma \ref{H-D1}
that there is a unitary operator $U=U_{\left(  A,G,\varepsilon \right)  }%
\in \mathcal{M}$ such that, for every $W\in \mathcal{B}$,%
\[
U\hat{\pi}\left(  W\right)  U^{\ast}=\hat{\rho}\left(  W\right)  \text{ .}%
\]
We therefore have,%
\[
\sum_{S\in A}\sum_{\gamma \in G}\left \Vert U\hat{\pi}\left(  S\right)  U^{\ast
}\left(  \gamma \right)  -\hat{\rho}\left(  S\right)  \left(  \gamma \right)
\right \Vert _{2,\beta_{\gamma}}%
\]%
\[
\leq \sum_{S\in A}\sum_{\gamma \in G}\left[  \left \Vert U\left(  \hat{\pi
}\left(  S\right)  \left(  \gamma \right)  -\hat{\pi}\left(  B_{S}\right)
\left(  \gamma \right)  \right)  U^{\ast}+\right \Vert _{2,\beta_{\gamma}%
}+\left \Vert \hat{\rho}\left(  B_{S}\right)  \left(  \gamma \right)  -\hat
{\rho}\left(  S\right)  \left(  \gamma \right)  \right \Vert _{2,\beta_{\gamma}%
}\right]
\]%
\[
\leq \varepsilon/37+\varepsilon/37<\varepsilon \text{.}%
\]
Also
\[
\sum_{S\in A}\sum_{\gamma \in G}\left \Vert \hat{\pi}\left(  S\right)  \left(
\gamma \right)  -U^{\ast}\hat{\rho}\left(  S\right)  U\left(  \gamma \right)
\right \Vert _{2,\beta_{\gamma}}%
\]%
\[
=\sum_{S\in A}\sum_{\gamma \in G}\left \Vert U\hat{\pi}\left(  S\right)
U^{\ast}\left(  \gamma \right)  -\hat{\rho}\left(  S\right)  \left(
\gamma \right)  \right \Vert _{2,\beta_{\gamma}}<\varepsilon \text{ .}%
\]
If we order the triples $\left(  A,G,\varepsilon \right)  $ by $\left(
\subset,\subset,>\right)  $, we have a net $\left \{  U_{\left(
A,G,\varepsilon \right)  }\right \}  $ of unitary operators in $\mathcal{M}$
such that, for every $T\in \mathcal{A}^{\# \#}$,
\[
U_{\left(  A,G,\varepsilon \right)  }\hat{\pi}(T)U_{\left(  A,G,\varepsilon
\right)  }^{\ast}\rightarrow \hat{\rho}(T)\text{ and }U_{\left(
A,G,\varepsilon \right)  }^{\ast}\hat{\rho}(T)U_{\left(  A,G,\varepsilon
\right)  }\rightarrow \hat{\pi}(T)\text{ .}%
\]
in the strong operator topology.
\end{proof}

\bigskip

\section{FWU algebras: A converse}

In this section we prove a converse of Theorem \ref{half1} when $\mathcal{A}$
is separable. We say that a unital C*-algebra $\mathcal{A}$ is an \emph{FWU}
\emph{algebra}, or that $\mathcal{A}$ has the \emph{finite weak*-uniqueness
property}, if, for every finite von Neumann algebra $\mathcal{M}$ with a
faithful normal tracial state $\tau$ and every pair $\pi,\rho:\mathcal{A}%
\rightarrow \mathcal{M}$ of unital $\ast$-homomorphisms such that, for all
$a\in \mathcal{A}$,
\[
\mathcal{M}\text{-rank}\left(  \pi \left(  a\right)  \right)  =\mathcal{M}%
\text{-rank}\left(  \rho \left(  a\right)  \right)  ,
\]
there is a net $\left \{  U_{i}\right \}  $ of unitary operators in
$\mathcal{M}$, such that, for every $a\in \mathcal{A}$,%
\[
\left \Vert U_{i}\pi \left(  a\right)  U_{i}^{\ast}-\rho \left(  a\right)
\right \Vert _{2,\tau}\rightarrow0\text{.}%
\]
Since every finite von Neumann algebra is a direct sum of algebras having a
faithful normal tracial state \cite{T}, being an FWU algebra is equivalent to
saying that for every finite von Neumann algebra and every pair $\pi
,\rho:\mathcal{A}\rightarrow \mathcal{M}$ of unital $\ast$-homomorphisms such
that, for all $a\in \mathcal{A}$,
\[
\mathcal{M}\text{-rank}\left(  \pi \left(  a\right)  \right)  =\mathcal{M}%
\text{-rank}\left(  \rho \left(  a\right)  \right)  ,
\]
we have that $\pi$ and $\rho$ are weak* approximately unitarily equivalent.

A key ingredient is a result of Alain Connes \cite{C}, who proved the
following characterization of hyperfiniteness. If $\mathcal{N}$ is a von
Neumann algebra, then the \emph{flip automorphism} $\pi:\mathcal{N}%
\otimes \mathcal{N}\rightarrow \mathcal{N}\otimes \mathcal{N}$ is the
automorphism defined by $\pi \left(  a\otimes b\right)  =b\otimes a$.

\begin{theorem}
\label{Connes}\cite{C}Suppose $\mathcal{N}\subset B\left(  H\right)  $ is a
$II_{1}$ factor von Neumann algebra acting on a separable Hilbert space. The
following are equivalent:

\begin{enumerate}
\item $\mathcal{N}$ is hyperfinite,

\item For every $n\in \mathbb{N}$, $x_{1},\ldots,x_{n}\in \mathcal{N}$,
$y_{1},\ldots,y_{n}\in \mathcal{N}^{\prime}$,%
\[
\left \Vert \sum_{k=1}^{n}x_{k}y_{k}\right \Vert _{H}=\left \Vert \sum_{k=1}%
^{n}x_{k}\otimes y_{k}\right \Vert _{H\otimes H}\text{ ,}%
\]

\item The flip automorphism $\pi$ on $\mathcal{N}\otimes \mathcal{N}$ is weak*
approximately unitarily equivalent in $\mathcal{N}\otimes \mathcal{N}$ to the
identity representation.
\end{enumerate}
\end{theorem}

\bigskip

If $\mathcal{N}$ is a von Neumann algebra and the flip automorphism $\pi$ is
weak* approximately equivalent to the identity, it easily follows that the
implementing net $\left \{  U_{\lambda}\right \}  $ of unitaries simultaneously
makes the maps $\rho_{1},\rho_{2}:\mathcal{N}\rightarrow \mathcal{N}%
\otimes \mathcal{N}$ by%
\[
\rho_{1}\left(  a\right)  =a\otimes1,\text{ }\rho_{2}\left(  a\right)
=1\otimes a\text{ for every }a\in \mathcal{N}\text{.}%
\]
weak* approximately equivalent. Connes' proof actually yields a seemingly
stronger statement. This statement fills in the details of Remark 7 in
\cite{DostH}.

\begin{theorem}
\label{Connes2}Suppose $\mathcal{N}\subset B\left(  H\right)  $ is a finite
factor von Neumann algebra acting on a separable Hilbert space $H$. Define
$\rho_{1},\rho_{2}:\mathcal{N}\rightarrow \mathcal{N}\otimes \mathcal{N}$ by%
\[
\rho_{1}\left(  a\right)  =a\otimes1,\text{ }\rho_{2}\left(  a\right)
=1\otimes a\text{ for every }a\in \mathcal{N}\text{.}%
\]
Suppose $\mathfrak{\rho}_{1}$ and $\rho_{2}$ are weak* approximately
equivalent in $\mathcal{N}\otimes \mathcal{N}$. Then $\mathcal{N}$ is hyperfinite.
\end{theorem}

\begin{proof}
Let $\tau$ be the unique faithful normal tracial state on $\mathcal{N}$. Then
$\tau \otimes \tau$ is a faithful normal tracial state on the factor
$\mathcal{N}\otimes \mathcal{N}\subset B\left(  H\otimes H\right)  $. Suppose
$\mathfrak{\rho}_{1}$ and $\rho_{2}$ are weak* approximately equivalent in
$\mathcal{N}\otimes \mathcal{N}$. Thus we can choose a net $\left \{
U_{\lambda}\right \}  $ of unitary operators in $\mathcal{N}\otimes \mathcal{N}$
such that, for every $a\in N,$%
\[
\left \Vert U_{\lambda}^{\ast}\left(  a\otimes1\right)  U_{\lambda}-\left(
1\otimes a\right)  \right \Vert _{2,\tau \otimes \tau}\rightarrow0.
\]
Suppose $n\in \mathbb{N}$, $x_{1},\ldots,x_{n}\in \mathcal{N}$, $y_{1}%
,\ldots,y_{n}\in \mathcal{N}^{\prime}$. Since $U_{\lambda}\in \mathcal{N}%
\otimes \mathcal{N}$ and each $1\otimes y_{k}\in \left(  \mathcal{N}%
\otimes \mathcal{N}\right)  ^{\prime}$, we have%
\[
U_{\lambda}^{\ast}\left(  \sum_{k=1}^{n}x_{k}\otimes y_{k}\right)  U_{\lambda
}=\sum_{k=1}^{n}U_{\lambda}^{\ast}\left(  x_{k}\otimes1\right)  \left(
1\otimes y_{k}\right)  U_{\lambda}=
\]%
\[
\sum_{k=1}^{n}\left[  U_{\lambda}^{\ast}\left(  x_{k}\otimes1\right)
U_{\lambda}\right]  \left(  1\otimes y_{k}\right)  \overset
{\text{\textrm{weak*}}}{\rightarrow}\sum_{k=1}^{n}\left(  1\otimes
x_{k}\right)  \left(  1\otimes y_{k}\right)  =1\otimes \left(  \sum_{k=1}%
^{n}x_{k}y_{k}\right)  .
\]
Since, for every $\lambda,$
\[
\left \Vert U_{\lambda}^{\ast}\left(  \sum_{k=1}^{n}x_{k}\otimes y_{k}\right)
U_{\lambda}\right \Vert =\left \Vert \sum_{k=1}^{n}x_{k}\otimes y_{k}\right \Vert
,
\]
it follows that
\[
\left \Vert \sum_{k=1}^{n}x_{k}y_{k}\right \Vert \leq \left \Vert \sum_{k=1}%
^{n}x_{k}\otimes y_{k}\right \Vert \text{ .}%
\]
It also follows that, for every $a\in \mathcal{N}$,%
\[
\left \Vert a\otimes1-U_{\lambda}\left(  1\otimes a\right)  U_{\lambda}^{\ast
}\right \Vert _{2,\tau \otimes \tau}=\left \Vert U_{\lambda}^{\ast}\left(
a\otimes1\right)  U_{\lambda}-\left(  1\otimes a\right)  \right \Vert
_{2,\tau \otimes \tau}\rightarrow0\text{ .}%
\]
Thus%
\[
U_{\lambda}\left[  1\otimes \left(  \sum_{k=1}^{n}x_{k}y_{k}\right)  \right]
U_{\lambda}^{\ast}=U_{\lambda}\sum_{k=1}^{n}\left(  1\otimes x_{k}\right)
\left(  1\otimes y_{k}\right)  U_{\lambda}^{\ast}%
\]%
\[
=\sum_{k=1}^{n}U_{\lambda}\left(  1\otimes x_{k}\right)  U_{\lambda}^{\ast
}\left(  1\otimes y_{k}\right)  \overset{\text{\textrm{weak*}}}{\rightarrow
}\sum_{k=1}^{n}\left(  x_{k}\otimes1\right)  (1\otimes y_{k})=\sum_{k=1}%
^{n}x_{k}\otimes y_{k}.
\]
Thus%
\[
\left \Vert \sum_{k=1}^{n}x_{k}\otimes y_{k}\right \Vert \leq \left \Vert
\sum_{k=1}^{n}x_{k}y_{k}\right \Vert \text{.}%
\]
Thus by Connes' theorem (Theorem \ref{Connes}), $\mathcal{N}$ is hyperfinite.
\end{proof}

\bigskip

We now prove our converse result.

\begin{theorem}
\label{half2}A separable unital C*-algebra is an FWU algebra if and only if it
is tracially nuclear.
\end{theorem}

\begin{proof}
Suppose $\mathcal{A}$ is an FWU algebra. Suppose $\tau$ is a factor tracial
state on $\mathcal{A}$. Let $\mathcal{N}=\pi_{\tau}\left(  \mathcal{A}\right)
^{\prime \prime}$. Since $\mathcal{A}$ is separable and $\pi_{\tau}$ has a
cyclic vector, $\mathcal{N}$ acts on a separable Hilbert space. If
$\mathcal{N}$ is finite-dimensional, then $\mathcal{N}$ is hyperfinite. Thus
we can assume that $\mathcal{N}$ is a $II_{1}$ factor. Then $\mathcal{N}%
\subset L^{2}\left(  \mathcal{A},\tau \right)  $ and $\pi_{\tau}\left(
\mathcal{A}\right)  $ is $\left \Vert {}\right \Vert _{2,\tau}$-dense in
$\mathcal{N}$. Define $\rho_{1},\rho_{2}:\mathcal{N}\rightarrow \mathcal{N}%
\otimes \mathcal{N}$ by%
\[
\rho_{1}\left(  b\right)  =b\otimes1,\text{ }\rho_{2}\left(  b\right)
=1\otimes b\text{ for every }b\in \mathcal{N}\text{.}%
\]
For $k=1,2,$ let $\sigma_{k}=\rho_{k}\circ \pi_{\tau}$:$\mathcal{A}%
\rightarrow \mathcal{N}\otimes \mathcal{N}$. Since $\left(  \tau \otimes
\tau \right)  \circ \rho_{1}=\left(  \tau \otimes \tau \right)  \circ \rho_{2},$ we
see that $\left(  \tau \otimes \tau \right)  \circ \sigma_{1}=\left(  \tau
\otimes \tau \right)  \circ \sigma_{2}.$ Since $\mathcal{A}$ is an FWU algebra,
$\sigma_{1}$ and $\sigma_{2}$ are weak* approximately unitarily equivalent in
$\mathcal{N}\otimes \mathcal{N}$. Thus there is a net $\left \{  U_{\lambda
}\right \}  $ of unitary operators in $\mathcal{N}\otimes \mathcal{N}$ such
that, for every $b\in \pi_{\tau}\left(  \mathcal{A}\right)  ,$%
\[
\left \Vert U_{\lambda}^{\ast}\left(  b\otimes1\right)  U_{\lambda}-\left(
1\otimes b\right)  \right \Vert _{2,\tau \otimes \tau}\rightarrow0\text{ .}%
\]
Since, for each $\lambda$, the map%
\[
b\mapsto U_{\lambda}^{\ast}\left(  b\otimes1\right)  U_{\lambda}-\left(
1\otimes b\right)
\]
is $\left \Vert {}\right \Vert _{2,\tau \otimes \tau}$-continuous and linear on
$\mathcal{N}$ and has norm at most $2$, and since $\pi_{\tau}\left(
\mathcal{A}\right)  $ is $\left \Vert {}\right \Vert _{2,\tau \otimes \tau}$ dense
in $\mathcal{N}$, we see that, for every $b\in \mathcal{N}$,%
\[
\left \Vert U_{\lambda}^{\ast}\left(  b\otimes1\right)  U_{\lambda}-\left(
1\otimes b\right)  \right \Vert _{2,\tau \otimes \tau}\rightarrow0\text{ .}%
\]
Thus, by Theorem \ref{Connes2}, $\mathcal{N}$ is hyperfinite. It follows from
Lemma \ref{factor} that $\mathcal{A}$ is tracially nuclear. The other
direction is contained in Theorem \ref{half1}.
\end{proof}

\end{document}